\documentclass[11pt]{amsart}
\usepackage{amsmath,amssymb,mathrsfs,color}
\usepackage{hyperref}
\hypersetup{colorlinks=true,
            linkcolor=blue,
            anchorcolor=blue,
            citecolor=blue}
\def\N{{\mathbb N}}

\def\R{{\mathbb R}}

\newtheorem{thm}{Theorem}[section]

\newtheorem{lem}[thm]{Lemma}
\newtheorem{prop}[thm]{Proposition}

\theoremstyle{definition}
\newtheorem{de}[thm]{Definition}
\theoremstyle{remark}
\newtheorem{rem}[thm]{Remark}

\numberwithin{equation}{section}
\allowdisplaybreaks

\begin{document}

\title[Invariant measures for stochastic Burgers equation]
{Invariant measures for stochastic Burgers equation in unbounded domains with space-time white noise}

%% First author
\author{Zhenxin Liu}
\address{Z. Liu: School of Mathematical Sciences,
Dalian University of Technology, Dalian 116024, P. R. China}
\email{zxliu@dlut.edu.cn}

%% Second author
\author{Zhiyuan Shi}
\address{Z. Shi (Corresponding author): School of Mathematical Sciences,
Dalian University of Technology, Dalian 116024, P. R. China}
\email{zyshi0@outlook.com; shizhiyuan@mail.dlut.edu.cn }
%\thanks{This work is partially supported by NSFC Grants 11271151, 11522104, and the startup and
%Xinghai Youqing funds from Dalian University of Technology.}

%%%%%%%%%%%%%%%%%%%%%%%%%%%%%%%%%%%

\date{January 16, 2025}

\makeatletter
\@namedef{subjclassname@2020}{\textup{2020} Mathematics Subject Classification}
\makeatother

\subjclass[2020]{35Q35, 35R60, 60H15}

\keywords{stochastic Burgers equation; space-time white noise; the entire real line; mild solution; existence of invariant measures; }

\begin{abstract}
In this paper, we investigate the stochastic damped Burgers equation with multiplicative space-time white noise defined on the entire real line. We prove the existence and uniqueness of a mild solution of the stochastic damped Burgers equation in the weighted space and establish that the solution is bounded in probability.
  Furthermore, by using the Krylov-Bogolioubov theorem, we obtain the existence of invariant measures.

\end{abstract}

\maketitle

\section{Introduction}
\setcounter{equation}{0}
The generalized stochastic Burgers equation is given as:
\begin{equation}\label{sde}
\dfrac{\partial u(t,x)}{\partial t}=\dfrac{\partial^2 u(t,x)}{\partial x^2}+f(x,u(t,x))
+\dfrac{\partial g}{\partial x}(x,u(t,x))+\sigma(x,u(t,x))\dfrac{\partial^2 W}{\partial t\partial x},\quad t>0,
\end{equation}
where $f,g$ and $\sigma$ are measurable real functions and $W=\{W(t,x), t>0, x\in \R\}$
represents a Brownian sheet defined on a complete probability space $(\Omega, \mathcal{F}, P)$.
For any $t>0$, $\mathcal{F}_{t}$ denotes the $\sigma$-field generated by the family of random variables $W=\{W(s,x), s\in[0,t], x\in \R\}$.
The family of $\sigma$-fields $\{\mathcal{F}_{t},t>0\}$ constitutes a stochastic basis on the probability space $(\Omega, \mathcal{F}, P)$. If $g= \frac{1}{2} u^{2}$, it is known as a stochastic Burgers equation, which is used as a simple model for turbulence. Over the past few decades, the stochastic Burgers equation has found wide-ranging applications in various fields, including fluid dynamics, statistical physics, cosmology, etc. \par

The well-posedness of equation \eqref{sde} is a prerequisite for investigating the existence of invariant measures. Specifically, Gy\"{o}ngy \cite{f} examined the Dirichlet problem associated with equation \eqref{sde} with space-time white noise and assumed square growth for $g$. Gy\"{o}ngy and Rovira \cite{h} further improved the conditions on $g$ by allowing polynomial growth, while also considering a more regular noise within the interval [0,1]. This adjustment was necessary as higher-order polynomial growth of $g$ posed technical incompatibilities with the space-time white noise.
Moreover, Gy\"{o}ngy and Nualart \cite{g} studied the existence and uniqueness of a mild solution for equation \eqref{sde} with multiplicative space-time white noise in an unbounded domain. They established the existence and uniqueness of a local solution by applying the fixed-point theorem. Subsequently, necessary estimates for global existence were obtained by introducing an auxiliary function. In the following discussion, to investigate the existence of invariant measures, we consider the problem in a weighted space due to the absence of compact embedding of the usual Sobolev spaces in unbounded domains.
However, the aforementioned studies do not take into account the weighted space. Therefore, we establish the existence and uniqueness of solutions to equation \eqref{sde} with $f=-k|u|u$, $g=\frac{1}{2}u^2$ in weighted $L^2$ space.\par

Recently, numerous studies focus on investigating the existence and uniqueness of invariant measures for equation \eqref{sde}. On the one hand, in the case of bounded domains, the Dirichlet problem associated with equation \eqref{sde} was extensively studied in \cite{q}. When $f=0$ and $g= \frac{1}{2} u^{2}$, Da Prato et al \cite{t}, Da Prato and G\c{a}tarek \cite{af} successfully established the existence and uniqueness of invariant measures. E et al \cite{e} studied the nonviscous case and considered the convergence of the invariant measures when the viscosity coefficient tends to 0. Moreover, Da Prato and Zabczyk proposed a method, documented in \cite{n}, for proving the existence of invariant measures based on the Krylov-Bogolioubov theorem. One the other hand, for unbounded domains, when $f=0$ and $g=0$,
Tessitore and Zabczyk \cite{i} conducted a study on the stochastic heat equation with spatially homogeneous noise in $\R^{d}, d\geq3$. It is worth noting that in order to address the lack of compactness, they examined the problem in a weighted space. Eckmann and Hairer \cite{a} considered equation \eqref{sde} with $f=u(1-|u|^2)$ and $g=0$. The noise is additive and white in time, coloured in space. They obtained smooth solutions and proved the existence of invariant measures. Assing and Manthey \cite{d} extended the results of \cite{i} by proving the existence of invariant measures for the case of $g=0$ in $\R^{d}$. The Gaussian noise is white in time and white or colored in space.  Misiats et al \cite{b} investigated equation \eqref{sde} with $g=0$ in $\R^{d}$, $d\geq3$, and expanded upon the results of \cite{d}, which may not satisfy the aforementioned condition of $f$ but satisfy $|f(x,u)|\leq \phi(x)\in L^{1}(\R^d)\cap L^{\infty}(\R^d)$. The noise is white in time and colored in space. The required dissipativity does not come from the nonlinear function $f$ but from the decaying property of the Green's function in three and higher dimensions in $\R^{d}$. For the case of $f=-\alpha u, \alpha\geq 0$, and $g$ in \eqref{sde} with polynomial growth, Kim \cite{o} examined the Cauchy problem for the stochastic Burgers equation \eqref{sde} using a more regular random additive noise, which is only white in time. When $f=0$ and $g=\frac{1}{2}u^2$, Dunlap et al \cite{j} investigated the invariant measures on the entire real line for equation \eqref{sde}, where the equation is forced by the derivative of a Gaussian noise that is white in time and smooth in space. Motivated by the above work, we will focus on studying stochastic damped Burgers equation in unbounded domains.\par

Specifically, in this paper, we explore equation \eqref{sde} with $f=-k|u|u$ and $g=\frac{1}{2}u^2$ on the real line with multiplicative space-time white noise.  The nonlinear term $f=-k|u|u$ is usually called the damping term. It describes various physical situations such as porous media flow, drag or friction effects. Several studies have also investigated such damping term \cite{gao, rao, song}. Due to the presence of convective term, the existence and uniqueness of solutions in the weighted space are more challenging than in usual $L^2$ space. Therefore, we add the damping term to control $L^2_{\rho}$-norm of the solution. Our investigation reveals that the equation has a unique mild solution in the weighted $L^2$ space. Additionally, we prove that the solution is bounded in probability  under suitable conditions. By utilizing the Krylov-Bogolioubov theorem and compactness property,
we demonstrate the tightness of the family of laws $\mathscr{L}(u(t))$ in the weighted $L^{2}$ space, thereby establishing the existence of invariant measures for equation \eqref{sde}.

The paper is organized as follows. In Section 2, we present some fundamental definitions, assumptions, and main results. In Section 3, we prove the existence and uniqueness of solutions. Finally, in section 4, we outline several crucial lemmas and obtain the existence of invariant measures.
\section{Preliminaries and main results}
\subsection{Preliminaries}
\setcounter{equation}{0}
Denote $\Delta$ the Laplace operator in the unbounded domain $\R$ and let $S(t)$ represent the semigroup generated by $\Delta$ on
${L}^{2}(\R)$ or on some subspaces of ${L}^{2}(\R)$ . For any $t>0$ and $\phi \in {L}^{2}(\R)$,\\
$$ S(t)\phi= G(t,x)\ast\phi,$$
where $\ast$ denotes the convolution operator and $G(t,x)=\cfrac{1}{\sqrt{4\pi{t}}}\exp\Big(-\cfrac{\vert{x}\vert^{2}}{4t}\Big)$.

Let $\rho(x)\geq 1$  be a continuous function. Assume that there exists $C^*>0$ such that $|\rho'(x)|\leq C^* \rho(x)$ on $\R$, and that for all $T>0$, there exists $C(T)>0$ such that
\begin{equation}\label{Cr}
G(t,\cdot)\ast\rho\leq {C_\rho(T)\rho}, \quad \forall t\in[0,T].
\end{equation}
It is obvious that $\rho(x)=(1+|x|^2)^r, r\geq1$ and $\rho(x)={\exp(m|x|)}, m>0$  satisfy the above requirements.

Define the weighted $L^p$ space
\begin{align*}
 L^{p}_{\rho}(\R):=\left\{\omega:\R\to \R,\int_{\R}\vert\omega(x)\vert^{p}\rho(x)dx<\infty\right\}
\end{align*}
with the norm
$$\|\omega\|_{p,\rho}^{p}:=\int_{\R}\vert\omega(x)\vert^{p}\rho(x)dx.$$
Note that $L^p_{\rho}(\R)\subset L^q_{\rho}(\R)$ and $L^p_{\rho}(\R)\subset L^{\infty}(\R)$ for $0<p<q$.

Consider the following equation
\begin{equation}\label{2.1}
\dfrac{\partial u(t,x)}{\partial t}=\Delta u(t,x)-k|u(t,x)|u(t,x)
-\frac{1}{2}\dfrac{\partial u^2(t,x)}{\partial x}+\sigma(x,u(t,x))\dfrac{\partial^2 W}{\partial t\partial x},
\end{equation}
$t>0$, $x\in \R$ with $u(0,x)=u_{0}(x)$. We impose some assumptions on diffusion coefficient $\sigma$.
\begin{enumerate}
  \item[($A1$)] $\sigma$ is a measurable real function satisfying the following conditions:
\begin{equation}\nonumber
\vert \sigma(x,r)\vert\leq b(x),
\end{equation}
\begin{equation}\nonumber
\vert \sigma(x,r)-\sigma(x,s)\vert \leq L\vert r-s\vert,
\end{equation}
for all  $x,r,s\in \R$, and for constant $L$ and nonnegative function $b(x)\in L_{\rho}^2(\R)$.
\end{enumerate}

 \begin{de}
We say that an $L^2_{\rho}(\R)$-valued and $\mathcal{F}_t$-adapted stochastic process $u(t,x)$ is a mild solution to \eqref{2.1} if for all $t\geq 0$ and almost all $x\in \R$,
\begin{align*}
u(t,x)=&\int_{\R}G(t,x-y)u_{0}(y)dy-k\int_{0}^{t}\int_{\R}G(t-s,x-y)|u(s,y)|u(s,y)dyds\\
&-\frac{1}{2}\int_{0}^{t}\int_{\R}G(t-s,x-y)\frac{\partial (u^2(s,y)) }{\partial y}dyds\\
&+\int_{0}^{t}\int_{\R}G(t-s,x-y)\sigma(u(s,y))W(dy,ds)\quad a.s.
\end{align*}
\end{de}

\subsection{Main results}
\begin{thm}\label{solution}
Assume that $(A1)$ holds and $k\geq\frac{C^*}{3}+\delta$. If the initial condition $u_{0}$ belongs to $L^2_{\rho}(\R)$, then there is a unique mild solution $u(t,x)$ to equation \eqref{2.1}. Moreover, for any $T>0$, there exists a constant $C(T)>0$ such that  for $p\geq 2$
 \begin{equation}
E \sup\limits_{t\in[0,T]}\|u(t,x)\|^p_{2,\rho}\leq C(T)(1+\|u_{0}\|^p_{2,\rho}).
 \end{equation}
\end{thm}

Let ${B}_{b}(L^2_{\rho}(\R))$ be the space of all bounded measurable functions on $L^2_{\rho}(\R)$ and $C_{b}(L^2_{\rho}(\R))$ be the space of all continuous and bounded functions on $L^2_{\rho}(\R)$. Then the semigroup $P_t$ associated with the solution $u(t,x)$ of \eqref{2.1} is defined by
\begin{align*}
P_{t}\phi(x)=E\phi(u(t, x))=\int_{L^2_{\rho}(\R)}P_{t}(x,dy)\phi(y),
\end{align*}
for any $t\geq 0$ and $\phi \in {B}_{b}$.
Its dual operator $P_{t}^{*}$ acting on the space $\mathcal{P}(L^2_{\rho}(\R))$ of all probability measures on $L^2_{\rho}(\R)$ is defined by
\begin{align*}
P_{t}^{*}\mu (\Gamma)=\int_{L^2_{\rho}(\R)}P_{t}(x,\Gamma)\mu(dx),
\end{align*}
for any $t\geq 0$, $\Gamma\in\mathcal{B}(L_{\rho}^2(\R))$ and $\mu\in\mathcal{P}(L^2_{\rho}(\R))$.
\begin{de}
The transition semigroup $P_{t}$ is Feller, if $P_{t}: C_{b}(L^2_{\rho}(\R))\to C_{b}(L^2_{\rho}(\R))$ for $t>0$.
\end{de}

\begin{de}
A probability measure $\mu\in \mathcal{P}(L^2_{\rho}(\R))$ is called an invariant measure of $(P_{t})_{t\geq0}$, if and only if $P_{t}^* \mu=\mu $
for all $t\geq 0$.
\end{de}

By applying the Krylov-Bogolioubov theorem,
we investigate the existence of invariant measures in an unbounded domain. In Section 3, the choice of $\rho(x)$ is more flexible, while in Section 4, we choose $\rho(x)=e^{m|x|}$  and $\hat{\rho}(x)=e^{\hat{m}|x|}$, where $m,\hat{m}>0$.
\begin{thm}\label{55}
 Assume that condition $(A1)$ holds. $(P_{t})_{t\geq0}$ is the Feller semigroup in $L^2_{\rho}(\R)$ associated with the solution $u(t,x)$ of equation \eqref{2.1}. Given that $u_{0}$ in $L^2_{\rho}(\R) {\cap} L_{\hat{\rho}}^2(\R)$, then for suitable $\rho(x)$, $\hat{\rho}(x)$ and $k$, there exists an invariant measure for $(P_{t})_{t\geq0}$ on $L^2_{\rho}(\R)$.
\end{thm}
\section{Solution: existence and uniqueness}
In this section, we prove the existence and uniqueness of solutions.
We have the following estimate about $G(t,x-y)$ (see e.g. \cite{ag}). For any $m, n\in\N\cup\{0\}$, there exist some constants $K, C>0$ such that
\begin{equation}\label{G}
\Big|\dfrac{\partial^m}{\partial t^m} \frac{\partial^n}{\partial y^n} G(t, x-y)\Big|
\leq K t^{-\frac{1+2 m+n}{2}} \exp{\Big(-\frac{C|x-y|^2}{t}\Big)},
\end{equation}
for any $ t >0$ and $x, y \in \mathbb{R}$. Then based on inequality \eqref{G}, we can find positive constants $K, C_1, C_2, C_3$ such that
\begin{equation}
\Big|\frac{\partial G}{\partial t}(t, x-y)\Big|  \leq K t^{-3 / 2} \exp \Big(-C_1 \frac{|x-y|^2}{t}\Big),
\end{equation}

\begin{equation}\label{88}
\Big|\frac{\partial G}{\partial y}(t, x-y)\Big| \leq K t^{-1} \exp \Big(-C_2 \frac{|x-y|^2}{t}\Big),
\end{equation}

\begin{equation}
\Big|\frac{\partial^2 G}{\partial y \partial t}(t, x-y)\Big|  \leq K t^{-2} \exp \Big(-C_3 \frac{|x-y|^2}{t}\Big),
\end{equation}
for all $t>0$, $x,y\in \R$. Moreover, in the following proof, C will denote a generic constant that may be different from one
formula to another.
\begin{lem}\label{eta2}
Define $$(G\sigma)(t,x)=\int_0^t\int_{\R}G(t-s,x-y)\sigma(u(s,y))W(ds,dy),$$ for $t\in[0,T]$ and $x\in \R$. Then we have
\begin{align*}
 E \sup\limits_{t\in [0,T]}\|G\sigma(t)\|_{2,\rho}^p\leq  C(T,  \|b\|_{2,\rho}^p), \quad p> 4.
\end{align*}
\end{lem}
\begin{proof}
Given $\alpha>0$, we can write
\begin{equation}
G\sigma(t,x)=\frac{ \sin \pi \alpha}{\pi}\int_0^t (t-\tau)^{\alpha-1}\int_{\R}G(t-\tau, x-z)Y(\tau, z)dzd\tau,
\end{equation}
where
 $$Y(\tau, z)=\int_{0}^{\tau}\int_{\R}(\tau-s)^{-\alpha}G(\tau-s,z-y)\sigma(u(s,y))W(ds,dy).$$
 Using Minkowski's inequality, for any $p>1$ and $\alpha>\frac{1}{p}$, we have
 \begin{align*}
 \|G\sigma(t)\|_{2,\rho}\leq &C\int_0^t  (t-\tau)^{\alpha-1}\left\|\int_{\R}G(t-\tau, x-z)Y(\tau, z)dz\right\|_{2,\rho}d\tau\\
 \leq &C\int_0^t  (t-\tau)^{\alpha-1} C^{\frac{1}{2}}_{\rho}(t)\|Y(\tau)\|_{2,\rho}d\tau\\
 \leq& C(T) \int_0^t  (t-\tau)^{\alpha-1}\|Y(\tau)\|_{2,\rho}d\tau\\
 \leq& C(T) \left(\int_0^t\|Y(\tau)\|_{2,\rho}^pd\tau\right)^{\frac{1}{p}}.
 \end{align*}
It follows that
 \begin{align*}
 \|G\sigma(t)\|_{2,\rho}^p\leq C(T) \int_0^t\|Y(\tau)\|_{2,\rho}^pd\tau.
 \end{align*}
 By Burkholder's inequality and condition (A1), we obtain
 \begin{align*}
 E\|Y(\tau)\|_{2,\rho}^p\leq&\left(\int_{\R}\int_0^\tau\int_{\R}(\tau-s)^{-2\alpha}G^2(\tau-s,z-y)\sigma^2(u(s,y))\rho(x)dydsdx\right)^\frac{p}{2}\\
 \leq& C(T, \|b\|_{2,\rho}^p)\left(\int_0^\tau(\tau-s)^{-2\alpha-1/2}ds\right)^{\frac{p}{2}}.
 \end{align*}
 Finally, for $\alpha<\frac{1}{4}$, we get
  \begin{align*}
 E \sup\limits_{t\in [0,T]}\|G\sigma(t)\|_{2,\rho}^p\leq &C(T, \|b\|_{2,\rho}^p) \int_0^T  \left(\int_0^\tau(\tau-s)^{-2\alpha-1/2}ds\right)^{\frac{p}{2}}d\tau\\
 \leq & C(T,  \|b\|_{2,\rho}^p).
 \end{align*}
 This completes the proof of Lemma \ref{eta2}.
\end{proof}
\subsection{Local existence and uniqueness}
Let ${B}(0,N)=\{u(t,x)\in L_{\rho}^2(\R): \|u(t)\|_{2,\rho}\leq N\}$ be the closed ball in $L_{\rho}^2(\R)$. Consider the mapping $\pi_N: L^2_{\rho}(\R)\to {B}(0,N)$ defined by
\begin{equation}
\pi_N(u)= \begin{cases}u, & \text { if }\|u\|_{2,\rho} \leq N, \\
 \frac{N}{\|u\|_{2,\rho}} u, & \text { if }\|u\|_{2,\rho}>N.\end{cases}
\end{equation}
Fix $N\in \mathbb{N}$, let us consider the equation
\begin{equation}\label{eN}
\begin{aligned}
\frac{\partial u^N(t, x)}{\partial t}= &\Delta u^N(t, x)-k|\pi_N u^N(t, x)|\pi_Nu^N(t,x)
-\frac{1}{2}\frac{\partial (\pi_N u^N(t, x))^2}{\partial x}\\
&+\sigma( \pi_N u^N(t,x))\frac{\partial ^2}{\partial t\partial x} W(t,x)
\end{aligned}
\end{equation}
with $u^N(0,x)=u_0(x)$.
\begin{lem}\label{lo}
Suppose that (A1) holds, $u_0\in L^2_{\rho}(\R)$ and $k\in \R$. Then for any fixed $N>0$, there is a unique mild solution $u^{N}(t,x)$ to equation \eqref{eN} such that $E(\sup\limits_{t\in[0,T]}\|u^N(t)\|_{2,\rho}^2)<\infty$.
\end{lem}
\begin{proof}
The proof will be done in three steps.

{\bf Step 1.} Suppose that $u=\{u(t), t\in[0,T]\}$ is an $L^2_{\rho}(\R)$-valued, $\mathcal{F}_t$-adapted stochastic process. Denote
\begin{equation}\label{A1}
\begin{aligned}
\mathcal{A}u^N=&\int_{\R}G(t,x-y)u_0(y)dy-k\int_0^t\int_{\R}G(t-s,x-y)\left(|\pi_N u^N|\pi_N u^N\right)(s,y)dyds\\
&+\frac{1}{2}\int_0^t\int_{\R}\frac{\partial G}{\partial y}(t-s,x-y)(|\pi_N u^N|^2)(s,y)dyds\\
&+\int_0^t\int_{\R}G(t-s,x-y)\sigma(\pi_N u^N)(s,y)W(dy,ds)\\
=:& \mathcal{A}_1(t,x)+\mathcal{A}_2u^N(t,x)+\mathcal{A}_3u^N(t,x)+\mathcal{A}_4u^N(t,x).
\end{aligned}
\end{equation}
We will prove that
\begin{equation}\label{Au}
E\left(\sup\limits_{t\in[0,T]}\|\mathcal{A}u^N(t)\|_{2,\rho}^2\right)<\infty.
\end{equation}
By using H\"{o}lder's inequality, we have
\begin{align*}
 \|\mathcal{A}_1(t,x)\|^2_{2,\rho}&\leq \int_{\R}\left(\int_{\R}G(t,x-y)dy\right)\left(\int_{\R}G(t,x-y)|u_0(y)|^2dy\right)\rho(x)dx\\
 &=\int_{\R}\int_{\R}G(t,x-y)|u_0(y)|^2\rho(x)dxdy\\
 &\leq C_{\rho}(T) \|u_0\|_{2,\rho}^2.
\end{align*}
Then by Minkowski's inequality and H\"{o}lder's inequality, we have
\begin{align*}
  \|\mathcal{A}_2u^N(t,x)\|_{2,\rho}^2 & \leq k \int_0^t \left\|\int_{\R}G(t-s,x-y) (|\pi_N u^N(s,y)|\pi_N u^N(s,y))dy\right\|_{2,\rho}^2ds\\
 & \leq k\int_0^t \int_{\R}\int_{\R}G(t-s,x-y) | \pi_N u^N (s,y) |^4 \rho(x)dxdyds\\
 &\leq kC_{\rho}(T) \int_0^t \| \pi_N u^N (s)\|_{2, \rho}^4ds\\
 &\leq C(k,N,T)t.
\end{align*}
Combining Minkowski's inequality, H\"{o}lder's inequality and inequality \eqref{88}, we get
\begin{align*}
&\|\mathcal{A}_3u^N(t,x)\|^2_{2,\rho}\\
\leq&\frac{1}{2}\left(\int_0^t\left\|\int_{\R}\frac{\partial G}{\partial y}(t-s,x-y)|\pi_Nu^N(s,y)|^2dy\right\|_{2,\rho}ds\right)^2\\
\leq& \frac{1}{2}\left(\int_0^t\left(\int_{\R}\frac{\partial G}{\partial y}(t-s,x-y)dy\right)^{\frac{1}{2}}\left(\int_{\R}\int_{\R}\frac{\partial G}{\partial y}(t-s,x-y)|\pi_Nu^N(s,y)|^4\rho(x)dx dy\right)^{\frac{1}{2}}ds\right)^2\\
\leq& C(T) \left(\int_0^t(t-s)^{-\frac{3}{4}}N^2ds\right)^2\\
=&C(N,T)t.
\end{align*}
Then by Lemma \ref{eta2}, we have
\begin{align*}
E\sup\limits_{t\in[0,T]}\|\mathcal{A}_4u^N(t,x)\|_{2,\rho}^2\leq C(T,  \|b\|_{2,\rho}^2).
\end{align*}
Hence, \eqref{Au} holds.

{\bf Step 2.}
Fix $\lambda>0$. Let $\mathcal{H}$ denote the Banach space of $L^2_{\rho}(\R)$-valued and $\mathcal{F}_t$-adapted stochastic process $u=\{u(t,x),t\in[0,T]\}$ such that $u(0)=u_0$, with the norm
$$\|u\|_{\mathcal{H}}^2=\int_0^Te^{-\lambda t}E\|u(t)\|_{2,\rho}^2dt<\infty.$$
Define the operator $\mathcal{A}$ on $\mathcal{H}$ by \eqref{A1}. For any $u\in \mathcal{H}$, by estimate \eqref{Au}, we have
\begin{align*}
\|\mathcal{A}u^N\|_{\mathcal{H}}^2=\int_{0}^T e^{-\lambda t}E\|\mathcal{A}u^N\|_{2,\rho}^2dt\leq C\int_{0}^T e^{-\lambda t}dt<\infty.
\end{align*}
Therefore, $\mathcal{A}$ is an operator that maps the Banach space $\mathcal{H}$ into itself.

{\bf Step 3.}
The following  argument will demonstrate that $\mathcal{A}$  is a contraction mapping on $\mathcal{H}$. Let $u,v\in \mathcal{H}$. By Minkowski's inequality, we have
\begin{align*}
&\|\mathcal{A}_2(u^N)(t)-\mathcal{A}_2(v^N)(t)\|_{2,\rho}^2\\
\leq& k^2\left\|\int_0^t\int_{\R}G(t-s,x-y)\left||\pi_N u^N(s,y)|\pi_N u^N(s,y)-|\pi_N v^N(s,y)|\pi_N v^N(s,y)\right|dyds\right\|_{2,\rho}^2\\
\leq& k^2\int_0^t \left\|\int_{\R}G(t-s,x-y)|\pi_N u^N(s,y)-\pi_N v^N(s,y)|\left(|\pi_N u^N(s,y)|+|\pi_N v^N(s,y)|\right)dy\right\|_{2,\rho}^2ds\\
\leq& k^2\left(\|\pi_N u^N(s)\|^2_{2,\rho}+\|\pi_N v^N(s,y)\|^2_{2,\rho}\right) \int_0^t \int_{\R}\int_{\R}G(t-s,x-y)|\pi_N u^N(s,y)-\pi_N v^N(s,y)|^2\rho(x)dxdyds\\
\leq &C(k,N) \int_0^t \int_{\R}C_{\rho}(T/2)\rho(y)|\pi_N u^N(s,y)-\pi_N v^N(s,y)|^2dy ds\\
\leq& C(k,N,T) \int_0^t \|u^N(s)-v^N(s)\|_{2,\rho}^2 ds.
\end{align*}
Applying Minkowski's inequality and inequality \eqref{88}, we obtain
\begin{align*}
&\|\mathcal{A}_3(u^N)(t)-\mathcal{A}_3(v^N)(t)\|_{2,\rho}^2\\
\leq&\left\|\int_0^t\int_{\R}\frac{\partial G}{\partial y}(t-s,x-y)(|\pi_N u^N(s,y)|^2-|\pi_N v^N(s,y)|^2)dyds\right\|_{2,\rho}^2\\
\leq &\left(\int_0^t\left\|\int_{\R}\frac{\partial G}{\partial y}(t-s,x-y)(|\pi_N u^N(s,y)|^2-|\pi_N v^N(s,y)|^2)dy \right\|_{2,\rho}ds\right)^2\\
\leq&4N^2\left( \int_0^t (t-s)^{-\frac{1}{2}}\left(\int_{\R}\frac{\partial G}{\partial y}(t-s,x-y)\left(|\pi_N u^N(s,y)|-|\pi_N v^N(s,y)|\right)^2dy\rho(x)dx\right)^{\frac{3}{4}}ds\right)^2\\
\leq& 4N^2 \left(\int_0^t (t-s)^{-\frac{3}{4}}\sqrt{C_{\rho}(T)}\|u^N(s)-v^N(s)\|_{2,\rho}ds\right)^2\\
\leq &C(N,T)\int_0^t(t-s)^{-\frac{3}{4}}\|u^N(s)-v^N(s)\|_{2,\rho}^2 ds.
\end{align*}
Using condition $(A1)$ and Burkholder's inequality, we deduce
\begin{align*}
&E\|\mathcal{A}_4(u^N)(t)-\mathcal{A}_4(v^N)(t)\|_{2,\rho}^2\\
\leq&  E\int_{\R}\left|\int_0^t\int_{\R}G(t-s,x-y)L|\pi_N u^N(s,y)-\pi_Nv^N(s,y)|W(dy,ds)\right|^2\rho(x)dx\\
\leq& E \int_{\R} \int_0^t\int_{\R}G^2(t-s,x-y)L^2|\pi_Nu^N(s,y)-\pi_Nv^N(s,y)|^2dyds\rho(x) dx\\
\leq& L^2E \int_0^t  \int_{\R}\int_{\R}G^2(t-s,x-y) |u^N(s,y)-v^N(s,y)|^2\rho(x)dxdyds\\
\leq& C(L,T) \int_0^t(t-s)^{-\frac{1}{2}}E\|u^N(s)-v^N(s)\|_{2,\rho}^2ds.
\end{align*}
From the previous estimates, we have
\begin{align*}
&\|\mathcal{A}(u^N)(t)-\mathcal{A}(v^N)(t)\|_{\mathcal {H}}^2\\
&=\int_0^T e^{-\lambda t}E\|\mathcal{A}(u^N)(t)-\mathcal{A}(v^N)(t)\|_{2, \rho}^2dt\\
&\leq C(k,N,T,L)\int_0^T e^{-\lambda t} \left(\int_0^t \left(1+(t-s)^{-\frac{3}{4}}+(t-s)^{-\frac{1}{2}}\right)E\|u^N(s)-v^N(s)\|_{2,\rho}^2ds\right)dt\\
&\leq C(k,N,T,L)\int_0^T\int_s^T e^{-\lambda t}\left(1+(t-s)^{-\frac{3}{4}}+(t-s)^{-\frac{1}{2}}\right)  E\|u^N(s)-v^N(s)\|_{2,\rho}^2dtds\\
&\leq C(k,N,T,L)\int_0^T\left(\int_0^{\infty}e^{-\lambda y}\left(1+y^{-\frac{3}{4}}+y^{-\frac{1}{2}}\right)dy\right) e^{-\lambda s}E\|u^N(s)-v^N(s)\|_{2,\rho}^2ds\\
&=C(k,N,T,L) \left(\int_0^{\infty}e^{-\lambda y} \left(1+y^{-\frac{3}{4}}+y^{-\frac{1}{2}}\right)dy\right)\|u^N(s)-v^N(s)\|_{\mathcal{H}}^2.
\end{align*}
By choosing $\lambda$ sufficiently large such that
$$C(k,N,T,L)\int_0^{\infty}e^{-\lambda y} \left(1+y^{-\frac{3}{4}}+y^{-\frac{1}{2}}\right)dy<1,$$
the operator $\mathcal{A}$ is a contraction mapping on $\mathcal{H}$. Therefore, there exists a unique fixed point of $\mathcal{A}$, which implies the existence of a unique mild solution to equation \eqref{eN}.
\end{proof}
\subsection{Global existence and uniqueness.}
Lemma \ref{lo} implies the local existence and uniqueness of a solution to equation \eqref{2.1}. The global existence will be proved in the following argument.

We recall that the solution to the linear problem
\begin{equation}
\left\{\begin{array}{l}
\frac{\partial z}{\partial t}=\Delta z+\sigma(u)\frac{\partial ^2 W}{\partial x\partial t}, \\
z(0)=0
\end{array}\right.
\end{equation}
is unique and is given by the so-called stochastic convolution
\begin{equation}\label{eta}
\eta(t,x)=\int_0^t\int_{\R}G(t-s,x-y)\sigma(u(s,y))W(dy,ds).
\end{equation}
By Lemma \ref{eta2}, we get
\begin{equation}\label{eta3}
 E \sup\limits_{t\in [0,T]}\|\eta(t)\|_{2,\rho}^p\leq  C( T,  \|b\|_{2,\rho}^p), p\geq2.
\end{equation}
Let $v=u-\eta$. Then $v$ is the solution of the following equation:
\begin{equation}\label{vvv}
\frac{\partial v}{\partial t}=\Delta v-k|v+\eta|(v+\eta)-\frac{1}{2} \frac{\partial}{\partial x}\left(v+\eta\right)^2,
\end{equation}
with $v(0)=u_0\in L_{\rho}^2(\R)$.

\begin{thm}\label{ve}
If $v\in C([0,T], L^2_{\rho}(\R))$ is the mild solution of \eqref{vvv}, then
\begin{equation}\label{esti}
\|v(t)\|_{2,\rho}^2\leq\left(\|u_{0}\|_{2,\rho}^2+tR_1(\eta)\right)e^{t\left(\frac{2(C^*)^2}{3}+\frac{C^*}{2}+k\right)},
\end{equation}
where
$$R_1(\eta)=\left(\frac{4k}{3\varepsilon_2^2}+\frac{2C^*}{3\varepsilon_3^2}\right)
\sup\limits_{s\in[0,T]}\|\eta(s)\|_{3,\rho}^3+\left(\frac{1}{2}C^*+1+k\right)\sup\limits_{s\in[0,T]}\|\eta(s)\|_{4,\rho}^4
+\frac{4}{3\varepsilon_4^2}\sup\limits_{s\in[0,T]}\|\eta(s)\|_{6,\rho}^6.$$
\end{thm}
 \begin{proof}
We can assume that $u_0$ and $\eta$ are regular. Multiplying \eqref{vvv} by $v\rho(x)$ and integrating over $\R$, we obtain
\begin{align*}
\frac{1}{2}\frac{d}{dt}\|v(t)\|_{2,\rho}^2=&\int_{\R}v(t)\Delta v(t) \rho(x)dx-k\int_{\R}|v(t)+\eta(t)|(v(t)+\eta(t))v(t)\rho(x)dx\\
&-\int_{\R}\frac{1}{2} \frac{\partial}{\partial x}\left(v(t)+\eta(t)\right)^2v(t)\rho(x)dx\\
=:&I_1+I_2+I_3.
\end{align*}
By Young's inequality, we get
\begin{equation}\label{fir}
\begin{aligned}
I_1=&-\int_{\R}v(t) v_{x}(t) \rho_{_{x}} dx-\int_{\R}|v_{x}(t)|^2\rho dx\\
\leq& C^*\int_{\R}|v_{x}(t)||v(t)|\rho dx-\int_{\R}|v_{x}(t)|^2\rho dx\\
\leq&\frac{C^*\varepsilon}{2}\int_{\R}|v_{x}(t)|^2\rho dx+\frac{C^*}{2\varepsilon}\int_{\R}|v(t)|^2\rho dx-\int_{\R}|v_{x}(t)|^2\rho dx.\\
=&(\frac{C^*\varepsilon}{2}-1)\|v_{x}(t)\|_{2,\rho}^2+\frac{C^*}{2\varepsilon}\|v(t)\|_{2,\rho}^2.
\end{aligned}
\end{equation}
For $I_2$, we have
\begin{equation}\label{I22}
\begin{aligned}
I_2=&-k\int_{\R}|v(t)+\eta(t)|v^2(t)\rho dx-k\int_{\R}|v(t)+\eta(t)|v(t)\eta(t)\rho dx\\
\leq&-k\int_{\R}(|v(t)|-|\eta(t)|)v^2(t)\rho dx+k\int_{\R}v^2(t)\eta(t)\rho dx+k\int_{\R}v(t)\eta^2(t)\rho dx\\
=&-k\int_{\R}|v(t)|^3\rho dx+2k\int_{\R}v^2(t)|\eta(t)|\rho dx+\frac{\varepsilon_1}{2}k\int_{\R}v^2(t)\rho dx+\frac{1}{2\varepsilon_1}k\int_{\R}\eta^4(t)\rho dx\\
\leq&-k\|v(t)\|_{3,\rho}^3+\frac{4k\varepsilon_2}{3}\|v(t)\|_{3,\rho}^3+\frac{2k}{3\varepsilon_2^2}\|\eta(t)\|_{3,\rho}^3+\frac{\varepsilon_1}{2}k\|v(t)\|_{2,\rho}^2
+\frac{1}{2\varepsilon_1}k\|\eta(t)\|_{4,\rho}^4.
\end{aligned}
\end{equation}
Applying the integration by parts formula, we obtain
\begin{align*}
I_3=&-\frac{1}{2}\int_{\R} \frac{\partial}{\partial x}(|v(t)+\eta(t)|^2-|v(t)|^2)v(t)\rho dx-\frac{1}{2}\int_{\R} \frac{\partial |v(t)|^2}{\partial x}v(t)\rho dx\\
=&\frac{1}{2}\int_{\R}(|v(t)+\eta(t)|^2-|v(t)|^2)v(t)\rho_{_{x}} dx+\frac{1}{2}\int_{\R}(|v(t)+\eta(t)|^2-|v(t)|^2)v_{x}(t)\rho dx\\&-\frac{1}{2}\int_{\R} \frac{\partial |v(t)|^2}{\partial x} v(t)\rho dx.
\end{align*}
For the first term on the right-hand side of the above equation, by H\"{o}lder's inequality and Young's inequality, we have
\begin{equation}\label{firs}
\begin{aligned}
&\frac{1}{2}\left|\int_{\R}\left(|v(t)+\eta(t)|^2-|v(t)|^2\right)v(t)\rho_{_{x}} dx\right|\\
\leq& \frac{1}{2}C^*\int_{\R}\left(2|v(t)|^2\eta(t)\rho+\eta^2(t)|v(t)|\rho\right) dx\\
\leq&C^*\int_{\R}|v(t)|^2\eta(t)\rho dx+\frac{1}{4} C^*\|\eta(t)\|_{4,\rho}^4+\frac{1}{4} C^*\|v(t)\|_{2,\rho}^2\\
\leq&\frac{2C^*\varepsilon_3}{3}\|v(t)\|_{3,\rho}^3+\frac{C^*}{3\varepsilon_3^2}\|\eta(t)\|_{3,\rho}^3+\frac{1}{4} C^*\|\eta(t)\|_{4,\rho}^4+\frac{1}{4} C^*\|v(t)\|_{2,\rho}^2 .
\end{aligned}
\end{equation}
For the second term, similarly, we have
\begin{equation}\label{seco}
\begin{aligned}
&\frac{1}{2}\left|\int_{\R}\left(|v(t)+\eta(t)|^2-|v(t)|^2\right)v_{x}(t)\rho dx\right|\\
\leq& \frac{1}{4}\int_{\R}(2|v(t)||\eta(t)|+\eta^2(t))^2\rho dx+\frac{1}{4}\|v_x(t)\|_{2,\rho}^2\\
\leq&2\int_{\R}v^2(t)\eta^2(t)\rho dx+\frac{1}{2}\|\eta(t)\|_{4,\rho}^4+\frac{1}{4}\|v_x(t)\|_{2,\rho}^2\\
\leq &\frac{4\varepsilon_4}{3}\|v(t)\|_{3,\rho}^3+\frac{2}{3\varepsilon_4^2}\|\eta(t)\|_{6,\rho}^6+\frac{1}{2}\|\eta(t)\|_{4,\rho}^4+\frac{1}{4}\|v_x(t)\|_{2,\rho}^2.
\end{aligned}
\end{equation}
For the third term, we have
\begin{equation}\label{thir}
\begin{aligned}
-\frac{1}{2}\int_{\R} \frac{\partial}{\partial x}|v(t)|^2v(t)\rho dx&=-\int_{\R}|v(t)|^2v_{x}(t)\rho dx\\
&\leq \frac{1}{3}C^* \int_{\R}|v(t)|^3\rho dx.
\end{aligned}
\end{equation}
Then, combining estimates \eqref{firs}, \eqref{seco} and \eqref{thir}, we get
\begin{equation}\label{thi}
\begin{aligned}
I_3\leq&\left(\frac{2C^*\varepsilon_3}{3}+\frac{4\varepsilon_4}{3}+\frac{1}{3}C^*\right) \|v(t)\|_{3,\rho}^3+\frac{1}{4}C^*\|v(t)\|_{2,\rho}^2+\frac{1}{4}\|v_{x}(t)\|_{2,\rho}^2\\
&+\frac{C^*}{3\varepsilon_3^2}\|\eta(t)\|_{3,\rho}^3+\left(\frac{1}{4}C^*+\frac{1}{2}\right)\|\eta(t)\|_{4,\rho}^4+\frac{2}{3\varepsilon_4^2}\|\eta(t)\|_{6,\rho}^6.
\end{aligned}
\end{equation}
Consequently, by inequalities \eqref{fir}, \eqref{I22} and \eqref{thi}, we have
\begin{equation}
\begin{aligned}
\frac{1}{2}\frac{d}{dt}\|v(t)\|_{2,\rho}^2\leq&\left(\frac{C^*\varepsilon}{2}-1+\frac{1}{4}\right)\|v_{x}(t)\|_{2,\rho}^2
+\left(\frac{C^*}{2\varepsilon}+\frac{\varepsilon_1}{2}k+\frac{1}{4}C^*\right)\|v(t)\|_{2,\rho}^2\\
&+\left(\frac{2C^*\varepsilon_3}{3}+\frac{4\varepsilon_4}{3}+\frac{1}{3}C^*+\frac{4k\varepsilon_2}{3}-k \right)\|v(t)\|_{3,\rho}^3\\
&+\left(\frac{1}{4}C^*+\frac{1}{2}+\frac{1}{2\varepsilon_1}k\right)\|\eta(t)\|_{4,\rho}^4+\left(\frac{2k}{3\varepsilon_2^2}+\frac{C^*}{3\varepsilon_3^2}\right)
\|\eta(t)\|_{3,\rho}^3+\frac{2}{3\varepsilon_4^2}\|\eta(t)\|_{6,\rho}^6.
\end{aligned}
\end{equation}
Let $\varepsilon=\frac{3}{2C^*}$, $\varepsilon_1=1$, $\delta=\frac{2C^*\varepsilon_3}{3}+\frac{4\varepsilon_4}{3}+\frac{4k\varepsilon_2}{3}$ and $k\geq \frac{C^*}{3}+\delta$. We have
\begin{align*}
\frac{d}{dt}\|v(t)\|_{2,\rho}^2\leq &\left(\frac{2(C^*)^2}{3}+k+\frac{1}{2}C^*\right)\|v(t)\|_{2,\rho}^2\\
&+\left(\frac{4k}{3\varepsilon_2^2}+\frac{2C^*}{3\varepsilon_3^2}\right)
\|\eta(t)\|_{3,\rho}^3+\left(\frac{1}{2}C^*+1+{k}\right)\|\eta(t)\|_{4,\rho}^4+\frac{4}{3\varepsilon_4^2}\|\eta(t)\|_{6,\rho}^6.
\end{align*}
 By Gronwall's inequality, we deduce
\begin{align*}
\|v(t)\|_{2,\rho}^2\leq& \left(\|u_{0}\|_{2,\rho}^2+\int_0^tR_1(\eta)ds\right)e^{t\left(\frac{2(C^*)^2}{3}+\frac{C^*}{2}+k\right)}\\
\leq &C\left(\|u_{0}\|_{2,\rho}^2+tR_1(\eta)\right),
\end{align*}
where
$$R_1(\eta)=\left(\frac{4k}{3\varepsilon_2^2}+\frac{2C^*}{3\varepsilon_3^2}\right)
\sup\limits_{s\in[0,T]}\|\eta(s)\|_{3,\rho}^3+\left(\frac{1}{2}C^*+1+k\right)\sup\limits_{s\in[0,T]}\|\eta(s)\|_{4,\rho}^4
+\frac{4}{3\varepsilon_4^2}\sup\limits_{s\in[0,T]}\|\eta(s)\|_{6,\rho}^6.$$
We have obtained a priori estimate for the regular $u_0$ and $\eta$, which, by taking the limit, leads to estimate \eqref{esti}.
\end{proof}

{\bf Proof of Theorem \ref{solution}.}

At first, we examine the uniqueness of the solution. Suppose that $u$ and $v$ are two solutions to equation \eqref{2.1}. For every natural  number $N$, we define the stopping time
$$
\sigma_N=\inf\{t\geq 0: \inf(\|u(t)\|_{2,\rho},\|v(t)\|_{2,\rho})\geq N\}\wedge T.
$$
Let $u_N(t)=u(t\wedge\sigma_N)$ and $v_N(t)=v(t\wedge\sigma_N)$ for all $t\in[0,T]$. Then $u_N(t)$ and $v_N(t)$ are solutions of equation \eqref{eN}. Hence $u_N(t)=v_N(t)$ a.s. for all $t\in [0,T]$. Taking the limit as $N\to \infty$, we conclude that $u(t)=v(t)$ a.s. for all $t\in [0,T]$.

To establish the existence of the solution for equation \eqref{2.1}, let $u^N$ be the solution to equation \eqref{eN} for any $N>0$. Consider the stopping time
$$
\tau_N=\inf\{t\geq 0: \|u(t)\|_{2,\rho}\geq N\}\wedge T.
$$
Notice that $u^M(t)=u^N(t)$ for $M\geq N$ and $t\leq \tau_N$. Therefore we can set $u(t)=u^N(t)$ if $t\leq \tau_N$ and then we have constructed a mild solution to equation \eqref{2.1} on the interval $[0,\tau_{\infty})$, where $\tau_{\infty}=\sup\limits_{N}\tau_N$. It remains to show that
$$
P(\tau_{\infty}=T)=1.
$$
Taking into account inequality \eqref{eta3} and inequality \eqref{esti}, we have
$$\sup\limits_{N}E(\sup\limits_{t\in[0,T]}\|u^N(t)\|_{2,\rho}^2)<\infty.$$
Since
\begin{equation}\label{tau}
\begin{aligned}
P(\tau_{N}<T)\leq &P\left(\sup\limits_{t\in[0,T]} \|u^N(t)\|_{2,\rho}^2\geq N^2\right)\\
\leq &\frac{E\left(\sup\limits_{t\in[0,T]}\|u^N(t)\|_{2,\rho}^2\right)}{ N^2}\leq\frac{C(1+\|u_0\|^2_{2,\rho})}{N^2},
\end{aligned}
\end{equation}
where $C$ is independent of $N$, it follows that $\tau_{\infty}=T$ a.s.
\section{Existence of invariant measures}

\subsection{Boundedness in probability}
 Different weights correspond to different $C_{\rho}(t)$  in inequality \eqref{Cr}. Our subsequent calculations heavily rely on the form of $C_{\rho}(t)$, so we choose $\hat{\rho}(x)=e^{\hat{m}|x|}$, where $\hat{m}>0$. Morover, in this section, the function $b(x)$ in condition (A1) belongs to $L^2_{\hat\rho}(\R)$. To establish the existence of invariant measures for equation \eqref{2.1}, we require the following estimates.
\begin{thm}\label{77}
Assume that (A1) holds. Let $0<C^*<\frac{3}{5}$ and $k>\max\{ \frac{C^*}{3}+\delta, \frac{4}{3-5C^*}\}$. Then for any $\varepsilon >0$, there exists $R(\varepsilon)>0$, such that for any $T\geq1$, we have
\begin{equation}\label{T}
\frac{1}{T}\int_{0}^{T}P(\|u(t,u_{0})\|_{2,\hat{\rho}}> R)\leq\varepsilon.
\end{equation}
\end{thm}
\begin{rem}
By Lemma \ref{z} and Lemma \ref{y}, it is straightforward to conclude that inequality \eqref{T} holds.
\end{rem}

\begin{lem}\label{Gt}
Let $\tilde{G}(t,x-y)=\dfrac{1}{\sqrt{t}}\exp({-\dfrac{|x-y|^2}{at}})$. Then we have
$$\int_{\R}\tilde{G}(t,x-y)\hat\rho(y)dy\leq C e^{\frac{\hat{m}^2 at}{4}}\hat\rho(x).$$
\end{lem}
\begin{proof}
\begin{align*}
\int_{\R}\tilde{G}(t,x-y)\hat\rho(y)\hat\rho^{-1}(x)dy
=&\int_{\R}\frac{1}{\sqrt{t}}e^{-\frac{|x-y|^2}{at}}e^{-\hat{m}|y|}e^{\hat{m}|x|}dy\\
\leq&\int_{\R}\frac{1}{\sqrt{t}}e^{-\frac{|x-y|^2}{at}}e^{\hat{m}|x-y|}dy\\
\leq&\int_{0}^\infty \frac{1}{\sqrt{t}}e^{-\frac{r^2}{at}}e^{\hat{m}r}dr=\int_{0}^\infty e^{-\frac{\tilde{r}^2}{a}+\hat{m}\sqrt{t}\tilde{r}}d\tilde{r}\\
=&e^{\frac{\hat{m}^2 at}{4}}\int_{0}^\infty e^{-\frac{y^2}{a}}dy=Ce^{\frac{\hat{m}^2 at}{4}}.
\end{align*}
\end{proof}
\begin{lem}\label{poi}
For $\hat{m}<\sqrt{2}$ and $u(t)\in C_0^{\infty}(\R)$, we have
\begin{equation}\label{poinc}
\int_{\R}|u_x(t)|^2 \hat\rho(x)dx\geq \tilde{C}\int_{\R} |u(t)|^2 \hat\rho(x)dx ,
\end{equation}
where $\tilde{C}>0$ is a finite constant.
\end{lem}
\begin{proof}
Let $\Omega=\R\backslash\bar{G}$, where $G$ is a bounded domain in $\R$ such that $0\in G$. Then by Example 5.5 in \cite{ed}, there exists a finite constant $C>0$ such that
$$\int_{\Omega} |u(t)|^2 \hat\rho(x)dx\leq C\int_{\Omega}|u_x(t)|^2 \hat\rho(x)dx.$$
Note that in $G$, for $u\in C_0^{\infty}(G)$, the following inequality holds:
$$\int_{G} |u(t)|^2 dx\leq \int_{G}|u_x(t)|^2 dx.$$
Therefore we have
\begin{align*}
\int_{G}|u(t)|^2 \hat\rho(x) dx=&\int_{G}(|u(t)| \hat\rho^{\frac{1}{2}})^2 dx\\
\leq &\int_{G}\left(u_x(t)\hat\rho^{\frac{1}{2}}+u(t) (\hat\rho^{\frac{1}{2}})_{x}\right)^2 dx\\
\leq& 2\int_{G}|u_x(t)|^2\hat\rho dx+\frac{\hat{m}^2}{2}\int_{G}u^2(t)\hat\rho dx.
\end{align*}
Since $\hat{m}<\sqrt{2}$,  it follows that
$$\int_{G}|u_x(t)|^2\hat\rho dx\geq\frac{2-m^2}{4} \int_{G}|u(t)|^2 \hat\rho dx.$$
Consequently, inequality \eqref{poinc} is deduced.
\end{proof}

We consider the modified Ornstein-Uhlenbeck equation
\begin{equation}\label{ou}
\frac{\partial z}{\partial t}=\Delta z(t,x)dt-\alpha z(t,x)dt+\sigma(u(t,x))\frac{\partial^2 W}{\partial x\partial t}
\end{equation}
with zero initial, where $\alpha>0$ and will be chosen later. It is known that the solution of equation \eqref{ou} is
\begin{equation}\label{ouz}
\eta_{\alpha}(t,x)=\int_0^t \int_{\R}e^{-\alpha(t-s)}G(t-s,x-y)\sigma(u(s))W(dy,ds).
\end{equation}
\begin{lem}\label{z}
Assume that condition $(A1)$ holds and $k>0$. Then for any $\alpha>\frac{\hat{m}^2}{4}$, there exists positive constant $C(\alpha)$ such that for any $t\geq 0$,
$$
E\|\eta_{\alpha}(t)\|_{2,\hat\rho}^p\leq C(\alpha), \quad p\geq 1.
$$
Moreover, \\
\centerline{$\lim\limits_{ \alpha\to +\infty}C(\alpha)=0$.}
\end{lem}
\begin{proof}
By Lemma \ref{Gt}, we have
\begin{align*}
E\|\eta_{\alpha}(t)\|_{2,\hat\rho}^p=&E\left(\int_{\R}\left|\int_0^t\int_{\R}G(t-s,x-y) e^{-\alpha(t-s)}\sigma(u(s,y))dW(s,y)\right|^2\hat\rho(x)dx\right)^{\frac{p}{2}}\\
\leq &E\left(\int_{\R}\int_0^t\int_{\R}G^2(t-s,x-y) e^{-2\alpha(t-s)}\sigma^2(u(s,y))dyds\hat\rho(x)dx\right)^{\frac{p}{2}}\\
\leq&C\|b\|_{2,\hat\rho}^p\left(\int_0^t (t-s)^{-\frac{1}{2}}e^{-2\alpha(t-s)}e^{\frac{\hat{m}^2}{2}(t-s)}ds\right)^{\frac{p}{2}}\\
=&C\left(\int_0^t (t-s)^{-\frac{1}{2}}e^{(\frac{\hat{m}^2}{2}-2\alpha)(t-s)}ds\right)^{\frac{p}{2}}.
\end{align*}
For $\alpha>\frac{\hat{m}^2}{4}$, by using the properties of the Gamma function, we obtain
\begin{equation}
E\|\eta(t)\|_{2,\hat\rho}^p\leq C\pi^{\frac{p}{4}}\left(2\alpha-\frac{\hat{m}^2}{2}\right)^{-\frac{p}{4}}=C(\alpha), \quad p\geq 1.
\end{equation}
\end{proof}

Combining Lemma \ref{z} with the Chebyshev inequality, we have
\begin{equation}\label{000}
\sup\limits_{t\geq{0}}P(\|\eta_{\alpha}(t)\|_{2,\hat{\rho}}^p>R)\leq\frac{C(\alpha)}{R^2}
\end{equation}
for any $R>0$.

Now we are going to find a similar result for the process $v_{\alpha}=u-\eta_{\alpha}$ which solves
\begin{equation}\label{vv}
\dfrac{\partial{v}(t)}{\partial t}=\Delta v(t)-k|v(t)+\eta_{\alpha}(t)|(v(t)+\eta_{\alpha}(t))-\frac{1}{2}\frac{\partial (v(t)+\eta_{\alpha}(t))^2}{\partial x}+\alpha \eta_{\alpha}(t)
\end{equation} with initial value $v(0)=0$. For this aim we have the following result.
\begin{lem}\label{y}
Assume that $0<\hat{m}<\min\left\{\frac{3\tilde{C}}{4\sqrt{\tilde{C}}+1}, \sqrt{2}\right\}$ and $k>\max\left\{ \frac{\hat{m}}{3}+\delta, \frac{4}{3\tilde{C}-\hat{m}-4\hat{m}\sqrt{\tilde{C}}}\right\}.$ Then for any $\varepsilon>0$, there exist $\alpha, R>0$ such that
\begin{equation}\label{T3}
\frac{1}{T}\int_{0}^{T}P(\|v_{\alpha}(t)\|^2_{2,\hat{\rho}}> R)ds<\varepsilon,
\end{equation}
for any $T>0$.
\end{lem}
\begin{proof}
Multiplying inequality \eqref{vv} by $v_{\alpha}\rho(x)$  and integrating over $\R$, we find
\begin{align*}
\frac{1}{2}\frac{d}{dt}\|v_{\alpha}(t)\|_{2,\rho}^2\leq&\left(\frac{\hat{m}\varepsilon}{2}-1+\frac{1}{4}\right)\|(v_{\alpha})_{x}\|_{2,\rho}^2+\left(\frac{\hat{m}}{2\varepsilon}
+\frac{\varepsilon_1k}{2}+\frac{\alpha \varepsilon_5}{2}+\frac{\hat{m}}{4}\right)\|v_{\alpha}(t)\|_{2,\rho}^2\\
&+\left(\frac{2\hat{m}\varepsilon_3}{3}+\frac{4\varepsilon_4}{3}+\frac{4k\varepsilon_2}{3}+\frac{\hat{m}}{3}-k \right)\|v_{\alpha}(t)\|_{3,\rho}^3+\frac{\alpha}{2\varepsilon_5}\|\eta_{\alpha}(t)\|_{2,\rho}^2\\
&+\left(\frac{2k}{3\varepsilon_2^2}+\frac{\hat{m}}{3\varepsilon_3^2}\right)
\|\eta_{\alpha}(t)\|_{3,\rho}^3+\left(\frac{\hat{m}}{4}+\frac{1}{2}+\frac{k}{2\varepsilon_1}\right)\|\eta_{\alpha}(t)\|_{4,\rho}^4
+\frac{2}{3\varepsilon_4^2}\|\eta_{\alpha}(t)\|_{6,\rho}^6.
\end{align*}
Take $\varepsilon=\frac{1}{\sqrt{\tilde{c}}}$, $\varepsilon_1=\frac{1}{k^2}$, $\varepsilon_5=\frac{1}{\alpha k}$, $\delta=\frac{2\hat{m}\varepsilon_3}{3}+\frac{4\varepsilon_4}{3}+\frac{4k\varepsilon_2}{3}$. Let $0<\hat{m}<\min\{\frac{3\tilde{C}}{4\sqrt{\tilde{C}}+1}, \sqrt{2}\}$ and $k>\max\{ \frac{\hat{m}}{3}+\delta, \frac{4}{3\tilde{C}-\hat{m}-4\hat{m}\sqrt{\tilde{C}}}\}$. Then we have
$$\frac{2\hat{m}\varepsilon_3}{3}+\frac{4\varepsilon_4}{3}+\frac{4k\varepsilon_2}{3}+\frac{\hat{m}}{3}-k \leq 0,$$
$$\frac{\hat{m}\varepsilon}{2}-1+\frac{1}{4}<0,$$ and
$$(\frac{\hat{m}\varepsilon}{2}-1+\frac{1}{4})\tilde{C}+\frac{\hat{m}}{2\varepsilon}
+\frac{\varepsilon_1k}{2}+\frac{\alpha \varepsilon_5}{2}+\frac{\hat{m}}{4}<0.$$
By combining Lemma \ref{poi}, we get
\begin{equation}
\begin{aligned}
\frac{d}{dt}\|v_{\alpha}(t)\|_{2,\rho}^2\leq&\left(2\hat{m}\sqrt{\tilde{C}}+\frac{\hat{m}}{2}+\frac{2}{k}-\frac{3\tilde{C}}{2}\right)\|v_{\alpha}(t)\|_{2,\rho}^2
+\hat{R}_{1}(\eta_{\alpha}),
\end{aligned}
\end{equation}
where
$$\hat{R}_{1}(\eta_{\alpha})={\alpha^2k}\|\eta_{\alpha}\|_{2,\rho}^2+\left(\frac{4k}{3\varepsilon_2^2}+\frac{2\hat{m}}{3\varepsilon_3^2}\right)
\|\eta_{\alpha}(t)\|_{3,\rho}^3+\left(\frac{ \hat{m}}{2}+1+k^3\right)\|\eta_{\alpha}(t)\|_{4,\rho}^4+\frac{4}{3\varepsilon_4^2}\|\eta_{\alpha}(t)\|_{6,\rho}^6.$$
For every $R > 0$ the following happens
\begin{align*}
&\frac{d}{d t} \ln (\|v_\alpha(t)\|_{2,\hat\rho}^2 \vee R) \\
&=1_{\{\|v_\alpha(t)\|_{2,\hat\rho}^2>R\}} \frac{1}{\|v_\alpha(t)\|_{2,\hat\rho}^2} \frac{d}{dt}\|v_\alpha(t)\|_{2,\hat\rho}^2 \\
&\leq 1_{\{\|v_\alpha(t)\|_{2,\hat\rho}^2>R\}}\left(2\hat{m}\sqrt{\tilde{C}}+\frac{\hat{m}}{2}+\frac{2}{k}-\frac{3\tilde{C}}{2}\right)+1_{\{\|v_\alpha(t)\|_{2,\hat\rho}^2>R\}} \frac{\hat{R}_1(\eta_{\alpha})}{\|v_\alpha(t)\|_{2,\hat\rho}^2}\\
&\leq 1_{\{\|v_\alpha(t)\|_{2,\hat\rho}^2>R\}}\left(2\hat{m}\sqrt{\tilde{C}}+\frac{\hat{m}}{2}+\frac{2}{k}-\frac{3\tilde{C}}{2}\right) + \frac{\hat{R}_1(\eta_{\alpha})}{R}.
\end{align*}
Integrating from 0 to $T$ and taking expectation, we get
\begin{align*}
&-\left(2\hat{m}\sqrt{\tilde{C}}+\frac{\hat{m}}{2}+\frac{2}{k}-\frac{3\tilde{C}}{2}\right)\int_0^T P(\|v_{\alpha}(s)\|_{2,\hat\rho}^2>R)ds\\
\leq &E\int_{0}^T 1_{\{\|v_\alpha(t)\|_{2,\hat\rho}^2>R\}} \frac{\hat{R}_1(\eta_{\alpha})}{R}dt\\
\leq&\frac{CT}{R}\sup\limits_{t\geq0}(C_{\alpha}+{\alpha^2}C_{\alpha})\\
\leq&CT\left(\frac{C_{\alpha}+\alpha^2C_{\alpha}}{R} \right).
\end{align*}
We can choose $\alpha$ and $R$ suitably such that $\int_0^T P(\|v_{\alpha}(s)\|_{2,\hat\rho}^2>R)ds$ can be as small as we want, uniformly in time.
\end{proof}

\subsection{The existence of invariant measures}
To employ the Krylov-Bogolioubov theorem, we will proceed to establish the Feller property of the semigroup $P_{t}$ associated with solution $u(t,x)$.
\begin{lem}\label{con}
For any $u_{01}$, $u_{02}\in L^2_{\rho}$ and $N>0$, define
\begin{equation}
\tau_{N}^{i}:=\inf\{t\geq0:\|u(t,u_{0i})\|_{2,\rho}>N\}, i=1,2,
\end{equation}
and
$$\tau_N=\tau_N^1\wedge\tau_N^2.$$ Then
$$ E\|u(t\wedge \tau_N,u_{01})-u(t\wedge \tau_N,u_{02})\|_{2,\rho}^2\leq C(t,N,k,L)\|u_{01}-u_{02}\|_{2,\rho}^2.$$
\end{lem}
\begin{proof}
Write $u_1(t,x)=u(t,u_{01})$ and $u_2(t,x)=u(t,u_{02})$. Set $t_{N}=t\wedge\tau_N$. From the definition of the mild solution, we have
\begin{align*}
u_1(t_N,x)-u_2(t_N,x)=&\int_{\R}G(t_N,x-y)(u_{01}(y)-u_{02}(y))dy\\
&+\int_0^{t_N}\int_{\R}kG(t_N-s,x-y)(|u_2(s,y)|u_2(s,y)-|u_1(s,y)|u_1(s,y))dyds\\
&-\frac{1}{2}\int_0^{t_N}\int_{\R}\frac{\partial G}{\partial y}(t_N-s,x-y)(u_1^2(s,y)-u_2^2(s,y))dyds\\
&+\int_0^{t_N}\int_{\R}G(t_N-s,x-y)(\sigma(u_1(s,y))-\sigma(u_2(s,y)))W(dy,ds).\\
=:&K_1(t_N,x)+K_2(t_N,x)+K_3(t_N,x)+K_4(t_N,x).
\end{align*}
As for $K_1(t_N,x)$, we have
\begin{align*}
\|K_1(t_N,x)\|_{2,\rho}^2\leq \int_{\R}\int_{\R}G(t_N, x-y)|u_{01}(y)-u_{02}(y)|^2\rho(x)dxdy\leq C(t)\|u_{01}-u_{02}\|_{2,\rho}^2.
\end{align*}
Applying Minkowski's inequality, we get
\begin{align*}
\|K_2(t_N,x)\|_{2,\rho}^2\leq& k \int_0^{t_N}\left\|\int_{\R}G(t_N, x-y) (|u_2(s,y)|u_2(s,y)-|u_1(s,y)|u_1(s,y)) dy\right\|_{2,\rho}^2ds\\
\leq& C(k,t)\|u_1(t)+u_2(t)\|_{2,\rho}^2\int_0^{t_N}\|u_1(s)-u_2(s)\|_{2,\hat\rho}^2ds\\
\leq& C(k,t,N) \int_0^{t_N}\|u_1(s)-u_2(s)\|_{2,\hat\rho}^2ds.
\end{align*}
By Minkowski's inequality and inequality \eqref{88}, we have
\begin{align*}
\|K_3(t_N,x)\|_{2,\rho}^2\leq&C\left( \int_0^{t_N}\left\|\int_{\R}\frac{\partial G}{\partial y}(t_N, x-y) (u^2_1(s,y)-u^2_2(s,y)) dy\right\|_{2,\rho}ds\right)^2\\
\leq& C(t)\|u_1(s)+u_2(s)\|_{2,\rho}^2 \left(\int_0^{t_N}(t_N-s)^{-\frac{3}{4}}\|u_1(s)-u_2(s)\|_{2,\hat\rho}ds\right)^2\\
\leq& C(t,N)\int_0^{t_N}(t_N-s)^{-\frac{3}{4}}\|u_1(s)-u_2(s)\|^2_{2,\hat\rho}ds.
\end{align*}
By  Burkholder's inequality and condition (A1), we obtain
\begin{align*}
E\|K_4(t_N,x)\|_{2,\rho}^2\leq&E\int_{\R} \left| \int_0^{t_N}\int_{\R}G^2(t_N, x-y) (\sigma(u_1(s,y))-\sigma(u_2(s,y)))^2 dyds\right|\rho(x)dx\\
\leq& C(L,t)\int_0^{t_N}(t_N-s)^{-\frac{1}{2}}E\|u_1(s)-u_2(s)\|_{2,\hat\rho}^2ds.
\end{align*}
Combining all the above estimates, we get
\begin{align*}
 &E\|u_1(t_N)-u_2(t_N)\|_{2,\rho}^2\\
 \leq &C(t)\|u_{01}-u_{02}\|_{2,\rho}^2\\
 &+ \int_0^{t_N} \left(C(k,t,N)+C(t,N)(t_N-s)^{-\frac{3}{4}}+C(L,t)(t_N-s)^{-\frac{1}{2}}\right)E\|u_1(s)-u_2(s)\|_{2,\hat\rho}^2ds.
 \end{align*}
Then by Gronwall's inequality, we derive the desired estimate.
\end{proof}
\begin{prop}
Suppose that $(A1)$ holds. Then $(P_{t})_{t\geq 0}$ associated with the solution $u(t,x,u_0)$ is a Feller semigroup.
\end{prop}
\begin{proof}
Denote ${B}(0,r)$ the closed ball $\{u(t,x)\in L_{\rho}^2(\R): \|u(t)\|_{2,\rho}\leq r\}$.
 Let $\phi\in C_{b}(L_{\rho}^2(\R))$.  It suffices to prove that for any $t>0$ and $r\in \mathbb{N}$,
\begin{equation}
\lim\limits_{\delta\to 0}\sup\limits_{u_{{01}},u_{{02}}\in B(0,r), \|u_{{01}}-u_{{02}}\|_{2,\rho}\leq \delta}|P_{t}\phi(t,u_{{01}})-P_{t}\phi(t,u_{{02}})|=0.
\end{equation}

For any $u_{{01}}, u_{02}\in {B}(0,r)$ and $N>r$,  as in Lemma \ref{con}, define
\begin{equation}
\tau_{N}^{i}:=\inf\{t\geq0:\|u(t,u_{0i})\|_{2,\rho}>N\}, i=1,2
\end{equation}
and
$$\tau_N=\tau_N^1\wedge\tau_N^2.$$
By inequality \eqref{tau}, we have
\begin{align}
&E|\phi(u(t,u_{0i}))-\phi(u(t\wedge\tau_{N},u_{0i}))| \\
&\leq 2|\phi|_{\infty}P(\tau_{N}<t)\nonumber\\
&\to 0  \qquad \text{as} \quad N\to +\infty. \nonumber
\end{align}
Then for any $\varepsilon>0$, choose  $ N>r$ sufficiently large such that for any $u_{0i}\in {B}(0,r),i=1,2$,
\begin{equation}\label{ph}
E|\phi(u(t,u_{0i}))-\phi(u(t\wedge\tau_{\hat N},u_{0i}))|\leq\varepsilon.
\end{equation}
For this $N$, since $\phi$ is uniformly continuous on ${B}(0,{\hat N})$, we can choose  $\eta>0$ such that for any $v,w\in {B}(0,{\hat N})$ with $\|v-w\|_{2,\rho}<\eta$,
we have
\begin{equation}\label{phi}
|\phi(v)-\phi(w)|\leq \varepsilon.
\end{equation}
By Lemma \ref{con}, we know that
\begin{equation}\label{ut}
\|u(t\wedge \tau_{\hat N},u_{01})-u(t\wedge \tau_{\hat N},u_{02})\|_{2,\rho}^2\leq c\|u_{01}-u_{02}\|_{2,\rho}^2.
\end{equation}
Thus for any $u_{01},u_{02}\in {B}(0,r)$ with $\|u_{01}-u_{02}\|_{2,\rho}^2<\dfrac{\delta^{2}\varepsilon}{2c|\phi|_{\infty}}$,  by inequality \eqref{phi} and inequality \eqref{ut}, we obtain
\begin{align*}
&E|\phi(u(t\wedge\tau_{\hat N},u_{01}))-\phi(u(t\wedge\tau_{\hat N},u_{02}))|\\
&=E\left[\phi(u(t\wedge\tau_{\hat N},u_{01}))-\phi(u(t\wedge\tau_{\hat N},u_{02}))I_{\{\|u(t\wedge\tau_{\hat N},u_{01})-u(t\wedge\tau_{\hat N},u_{02})\|_{2,\rho}\leq \delta\}} \right]\\
&+E\left[\phi(u(t\wedge\tau_{\hat N},u_{01}))-\phi(u(t\wedge\tau_{\hat N},u_{02}))I_{\{\|u(t\wedge\tau_{\hat N},u_{01})-u(t\wedge\tau_{\hat N},u_{02})\|_{2,\rho}> \delta\}} \right]\\
&\leq \varepsilon +2|\phi|_{\infty}P(\|u(t\wedge\tau_{\hat N},u_{01})-u(t\wedge\tau_{\hat N},u_{02})\|_{2,\rho}> \delta)\\
&\leq \varepsilon +2|\phi|_{\infty}\frac{1}{\delta^2}E(\|u(t\wedge\tau_{\hat N},u_{01})-u(t\wedge\tau_{\hat N},u_{02})\|_{2,\rho}^{2})\\
&\leq 2\varepsilon.
\end{align*}
According to inequality \eqref{ph} and inequality \eqref{ut}, we have
\begin{align*}
&E|\phi(u(t,u_{01}))-\phi(u(t,u_{02}))|\\
&\leq E|\phi(u(t,u_{01}))-\phi(u(t\wedge\tau_{\hat N},u_{01}))|+E|\phi(u(t\wedge\tau_{\hat N},u_{01}))-\phi(u(t\wedge\tau_{\hat N},u_{02}))|\\
&+E|\phi(u(t\wedge\tau_{\hat N},u_{02}))-\phi(u(t,u_{02}))|\\
&\leq 4\varepsilon,
\end{align*}
which completes the proof.
\end{proof}

\begin{prop}\label{s}
For $\rho(x)=e^{m|x|}$, $m>0$,  the restriction of semigroup $S(t)$  to $L^2_{\rho}$ is a $C_0$ semigroup. Moreover, let $\hat{\rho}(x)=e^{\hat{m}|x|}$. If $m<\hat{m}$, then for all $t>0$, $S(t)$ is compact from $L_{\hat{\rho}}^2(\R)$ to $L_{\rho}^2(\R)$.
\end{prop}
\begin{proof}
By Proposition 2.1 of \cite{i}, for all $\phi\in L^2(\R)$, we have
$$\left\|S(t)\phi \right\|_{2,\hat\rho}^{2}\leq C_{\hat\rho}(T)\|\phi\|_{2,\hat\rho}^2.$$
Therefore the restriction of $S(t)$ to $L^2_{\hat\rho}(x)$ is a bounded linear map. Moreover, $S$ is a $C_0$ semigroup on $L^2_{\hat\rho}$.

If $m<\hat{m}$, then we have
 $$ \int_{\R}\frac{\rho(x)}{\hat\rho(x)}dx= \int_{\R}e^{(m-\hat{m})|x|}dx<+\infty.$$
 By Proposition 2.1 of \cite{i}, we also know that $S(t)$ is a Hilbert-Schmidt operator, therefore, it is compact from $L^2_{\hat\rho}$ to $L^2_{\rho}$.
\end{proof}

\begin{prop}\label{ss}
Assume that $0<m<\hat{m}$. Denote $\hat{S}(t)\psi(y):=\displaystyle\int_{\R}\frac{\partial{G(t,x-y)}}{\partial y}\psi(y)dy$. Then for all $t>0$, $\hat{S}(t)$ is compact from $L_{\hat{\rho}}^2(\R)$ to $L_{\rho}^2(\R)$.
\end{prop}
\begin{proof}
The proof is similar to proposition \ref{s}. It is sufficient to prove that $\hat{S}(t)$ is a Hilbert-Schmidt operator. Let $\{e_i\}$ be an orthonormal basis in $L^2$, then we have
\begin{align*}
\sum_{i=1}^{\infty} \big\|\hat{S}(t)\frac{e_{i}}{\sqrt{\hat\rho}}\big\|_{2,\rho}^{2}
&=\sum_{i=1}^{\infty}\int_{\R}\rho(x)\left(\int_{\R}\frac{\partial{G(t,x-y)}}{\partial y}\frac{|e_{i}(y)|}{\sqrt{\hat\rho(y)}}dy\right)^2dx\\
&\leq Ct^{-\frac{1}{2}}\int_{\R}\int_{\R}\frac{\partial{G}}{\partial y}(t,x-y)\hat{\rho}^{-1}(y)\rho(x)dxdy\\
&\leq C_{\hat\rho}(t)t^{-1}\int_{\R}\hat{\rho}^{-1}(x)\rho(x)dx\\
&=C_{\hat\rho}(t)t^{-1}\int_{\R}e^{(m-\hat{m})|x|}dx<+\infty.
\end{align*}
The proof is complete.
\end{proof}
\begin{lem}\label{Mal}\rm{(\cite{i}, Theorem 3.1 (Step 1))}.
 Let $q>1$, $\alpha\in(\frac{1}{q},1]$ and $\phi\in L^{q}(0,1;L_{\hat\rho}^2(\R))$. Denote \\
\centerline {$M_{\alpha}\phi:=\displaystyle\int_{0}^{1}(1-s)^{\alpha-1}S(1-s)\phi(s)ds,$}
\centerline{$\hat M_{\alpha}\phi:=\displaystyle\int_{0}^{1}(1-s)^{\alpha-1}\hat S(1-s)\phi(s)ds.$}
Then $M_{\alpha}$ and $\hat M_{\alpha}$ are both compact operators from $L^{q}(0,1;L_{\hat\rho}^2(\R))$ to $L^2_{\rho}(\R)$.
\end{lem}

As a consequence, for any $r>0$ and $\alpha \in (\frac{1}{q},1]$,\\
\centerline{$K(r)=\{S(1)u_{0}+M_{1}\phi_{1}+\hat{M}_{1}\phi_{2}+M_{\alpha}\phi_{3}:\|u_{0}\|_{2,\hat{\rho}}<r, \|\phi_{i}\|_{L^{q}([0,1];L^2_{\hat{\rho}})}<r, i=1,2,3\}$}
is relatively compact in $L_{\rho}^2(\R)$.\par
\begin{lem}\label{Yt}
 For $\alpha\in (1/q, 1/4)$, we define
 $$Y_{\alpha}(t,x):=\int_{0}^{t}\int_{\R}(t-s)^{-\alpha}G(t-s,x-y)\sigma(u(s,y))W(ds,dy).$$ Then we have
\begin{equation}\label{t}
E\|k|u(t)|u(t)\|_{L^{q}(0,1;L^2_{\hat{\rho}})}^{q}=E\int_{0}^{1}\|k|u(t)|u(t)\|_{2,\hat{\rho}}^{q}dt\leq C\left(1+\|u_{0}\|_{2,\hat{\rho}}^q\right)
\end{equation}
and
\begin{equation}\label{x}
E\|Y_{\alpha}(t)\|_{L^{q}(0,1;L^2_{\hat{\rho}})}^{q}=\int_{0}^{1}\|Y_{\alpha}(t)\|_{2,\hat{\rho}}^{q}dt\leq \hat C,
\end{equation}
where $C$ and $\hat C$ are positive constants.
\end{lem}
\begin{proof}
By Theorem \ref{solution}, we have
\begin{align*}
E\|k|u(t)|u(t)\|_{L^{q}(0,1;L^2_{\hat{\rho}})}^{q}&=E\int_{0}^{1}\left(\int_{\R}(k|u(t,x)|^2)^{2}\hat{\rho}(x)dx\right)^{q/2}dt\\
&\leq k^{q} \int_0^1 E \|u(t)\|_{4,\hat{\rho}}^qdt\leq  k^{q} \int_0^1 E \|u(t)\|_{2,\hat{\rho}}^qdt\leq C\left(1+\|u_{0}\|_{2,\hat{\rho}}^q\right).
\end{align*}
Applying Burkholder's inequality, Young's inequality, condition $(A1)$ and $\alpha<1/4$, we get
\begin{align*}
E\|Y_{\alpha}(t)\|_{L^{q}(0,1;L^2_{\hat\rho})}^{q}&=E\int_{0}^{1}\left(\int_{\R}|Y_{\alpha}(t,x)|^{2}\hat{\rho}(x)dx\right)^{q/2}dt\\
&\leq C E\int_{0}^{1} \left(\int_{\R}\int_{0}^{t}(t-s)^{-2\alpha}\int_{\R}G^{2}(t-s,x-y)\sigma^{2}(u(s,y))dyds\hat{\rho}(x)dx\right)^{\frac{q}{2}}dt\\
&\leq C \int_0^1\left( \int_0^t(t-s)^{-2\alpha}(t-s)^{-\frac{1}{2}}e^{\frac{\hat{m}^2(t-s)}{2}}\|b\|_{2,\hat\rho}^2ds\right)^{\frac{q}{2}}dt\\
&\leq C \|b\|_{2,\hat\rho}^q e^{\frac{\hat{m}^2q}{4}}\int_0^1\left(\int_0^t (t-s)^{-2\alpha-\frac{1}{2}}e^{-\frac{\hat{m}^2s}{2}}ds \right)^{\frac{q}{2}}dt\\
&\leq C \|b\|_{2,\hat\rho}^q e^{\frac{\hat{m}^2q}{4}}\left(\int_0^1 t^{-2 \alpha-\frac{1}{2}}dt\right)^{\frac{q}{2}} \int_0^1 e^{-\frac{\hat{m}^2qt}{4}}dt\leq \hat{C}.
\end{align*}
This completes the proof of Lemma \ref{Yt}.
\end{proof}
{\bf Proof of Theorem \ref{55}.}
\begin{proof}
By employing the factorization formula, we have
\begin{equation}
u(1,u_{0})=S(1)u_{0}-M_{1}k|u(t)|u(t)+\frac{1}{2}\hat M_{1}u^2(t)+\displaystyle\frac{\sin\alpha\pi}{\pi}M_{\alpha}Y_{\alpha}(t).
\end{equation}
By Chebyshev's inequality, together with inequality \eqref{t} and inequality \eqref{x}, we deduce that there exists a constant $C>0$ such that for any $r>0$ and all $\|u_0\|_{2,\hat\rho}\leq r$,
\begin{align*}
&P(u(1,u_{0})\notin K(r))\\
\leq& P\left\{\|k|u(t)|u(t)\|_{L^{q}(0,1;L^2_{\hat{\rho}})}>r\right\}+P\left\{\|u^2(t)\|_{L^{q}(0,1;L^2_{\hat{\rho}})}>r\right\}+P\left\{\|Y_{\alpha}(t)\|_{L^{q}(0,1;L^2_{\hat{\rho}})}>\frac{\pi r}{\sin\alpha\pi}\right\}\\
\leq & \dfrac{1}{r^q}E\|k|u(t)|u(t)\|_{L^{q}(0,1;L^2_{\hat{\rho}})}^{q}+\dfrac{1}{r^q}E\|u^2(t)\|_{L^{q}(0,1;L^2_{\hat{\rho}})}^{q}
+\dfrac{\sin^{q}\alpha\pi}{(\pi r)^q}E\| Y_\alpha(t)\|_{L^{q}(0,1;L^2_{\hat{\rho}})}^q\\
\leq&C r^{-q}\big(1+\|u_{0}\|_{2,\hat{\rho}}^{q}\big).
\end{align*}
For any $t>1$, by the proof of Theorem 6.1.2 in \cite{n}, for $r>r_1>0$, we have
\begin{equation}\label{Kr}
\begin{aligned}
P(u(t)\in K(r))\geq\big(1-C r^{-q}(1+r_{1}^{q})\big)P\big(\|u(t-1)\|_{2,\hat{\rho}}\leq r_{1}\big).
\end{aligned}
\end{equation}
By integrating both sides of inequality \eqref{Kr} from 0 to $T$, we get
$$
\frac{1}{T}\int_{1}^{T+1}P\left(u(t)\in K(r)\right)dt
\geq\big(1-Cr^{-q}(1+r_{1}^{q})\big)\frac{1}{T}\int_{0}^{T}P\big(\|u(t)\|_{2,\hat{\rho}}\leq r_{1}\big)dt,\quad T\geq 1.
$$
Taking first $r_1$ and then $r>r_{1}$ be sufficiently large, combining Theorem \ref{77}, we obtain that\\
\centerline{$\dfrac{1}{T}\displaystyle\int_{1}^{T+1}\mathcal{L}(u(t))dt=\dfrac{1}{T}\displaystyle\int_{1}^{T+1}P\left(u(t)\in K(r)\right)dt$} is tight.
Hence there exists a sequence $ T_{n} \to +\infty$ such that\\
\centerline{$\dfrac{1}{T_{n}}\displaystyle\int_{1}^{T_{n}+1}\mathcal{L}(u(t))dt$}
converges weakly to $\mu$. By the Krylov-Bogolioubov theorem, $\mu$ is invariant for equation \eqref{2.1} in $L^2_{\rho}(\R)$. The proof is complete.
\end{proof}

\section*{Acknowledgements}
This work is supported by National Key R\&D Program of
China (No. 2023YFA1009200), NSFC (Grant 11925102), and LiaoNing Revitalization Talents Program (Grant XLYC2202042).

\end{document}